\documentclass[10pt]{article}

\title{The lattice points of a stacky polytope}
\author{Hironori Sakai}
\date{}

\usepackage{fourier}
\usepackage[T1]{fontenc}

\usepackage{amsmath,amssymb,amsthm}
\usepackage[all]{xy} 
\usepackage[pdftex]{hyperref}

\newtheorem{thm}{Theorem}[section]
\newtheorem{prop}[thm]{Proposition}
\newtheorem{lem}[thm]{Lemma}

\theoremstyle{definition}
\newtheorem{dfn}[thm]{Definition}
\newtheorem{assumption}[thm]{Assumption}

\theoremstyle{remark}
\newtheorem{remark}[thm]{Remark}

\newcommand{\squashup}[0]{\vspace{-3pt}}

\renewcommand{\epsilon}[0]{\varepsilon}
\renewcommand{\phi}[0]{\varphi}

\renewcommand{\tilde}[1]{\widetilde{#1}}

\renewcommand{\emph}[1]{\textbf{#1}}

\newcommand{\B}{\mathfrak{B}}
\newcommand{\C}{\mathbb{C}}

\renewcommand{\L}{\mathcal{L}}

\renewcommand{\P}{\mathcal{P}}
\newcommand{\Q}{\mathcal{Q}}
\newcommand{\R}{\mathbb{R}}
\newcommand{\T}{\mathbb{T}}

\newcommand{\X}{\mathcal{X}}
\newcommand{\Y}{\mathcal{Y}}
\newcommand{\Z}{\mathbb{Z}}
\def\i{\mathtt   i}  
\def\g{\mathfrak g}  

\newcommand{\Diff}{\mathfrak{Diff}}

\newcommand{\St}{\mathrm{St}}

\newcommand{\BG}{\mathfrak{B}G}
\newcommand{\BH}{\mathfrak{B}H}

\newcommand{\Cinf}{\mathcal{C}^\infty}

\newcommand{\Omegabas}{\Omega_\mathrm{bas}}

\newcommand{\acted}[0]{\curvearrowleft}

\DeclareMathOperator{\Hom}{Hom}
\DeclareMathOperator{\id}{id}

\DeclareMathOperator{\Spec}{Spec}
\DeclareMathOperator{\src}{src}
\DeclareMathOperator{\tgt}{tgt}

\DeclareMathOperator{\ch}{ch}
\DeclareMathOperator{\Td}{Td}

\newcommand{\GL}[0]{\mathrm{GL}}

\newcommand{\Ob}[1]{\mathrm{Ob}(#1)}
\newcommand{\Ar}[1]{\mathrm{Ar}(#1)}

\newcommand{\pt}[0]{\mathrm{pt}}

\newcommand{\pr}[0]{\mathrm{pr}}
\newcommand{\dual}[0]{{\scriptscriptstyle\vee}}

\newcommand{\Lie}[0]{\mathrm{Lie}}
\newcommand{\lm}[0]{\mu^{-1}(\tau)}

\newcommand{\pair}[1]{\langle#1\rangle}
\renewcommand{\bar}[1]{\overline{#1}}
\newcommand{\DG}[0]{\mathrm{DG}}
\newcommand{\Xt}[0]{\X_\tau}
\newcommand{\XD}[0]{\X_\Delta}

\newcommand{\omegastd}[0]{\omega_\mathrm{std}}
\newcommand{\vb}[1]{\mathbf{#1}}
\newcommand{\CP}[0]{\mathbb{CP}}

\usepackage{longtable}
\usepackage{bbold}
\usepackage[utf8]{inputenc}
\usepackage{array}
\begin{document}

\maketitle

\begin{abstract}
 We study the relation between the lattice points of a stacky polytope and a prequantisation of the stack associated to the stacky polytope. We introduce a prequantisation of a Deligne--Mumford stack and discuss the uniqueness and the existence of a prequantisation. After that we describe explicitly the condition for the existence of a prequantisation in terms of stacky polytope under some conditions and discuss the relation between the holomorphic line bundle associated to a prequantisation and the lattice points of the stacky polytope.
\end{abstract}

\section{Introduction} 
\label{sec:introduction}

It is well-known that the dimension of the quantisation space of a compact (quantisable symplectic) toric manifold agrees with the number of the lattice points of the moment polytope. This theorem is proved by Danilov in terms of algebraic geometry \cite{danilov78:_geomet}. A very simple proof of the theorem is also known \cite{hamilton08}.

Consider a complex line bundle $L$ of $M$ together with a connection whose curvature form agrees with the symplectic form of $M$. Such a connection gives a holomorphic structure to the line bundle and the quantisation space of $M$ is defined as the virtual vector space $\sum_q (-1)^qH^q(M,L)$. By the Demazure vanishing theorem $H^q(M,L)$ vanishes for $q>0$.

On the other hand, a toric manifold $M$ admits a moment map $\mu:M \to \R^d$ ($2d:=\dim M$) and its image $\Delta$ is called the moment polytope of $\mu$. The statement mentioned at the beginning says that the identity 
\[
 \dim H^0(M,L) = \sharp(\Delta \cap \Z^d)
\]
holds.

The theorem does not extend to (compact) symplectic toric Deligne--Mumford stacks \cite{lerman12:_hamil_delig_mumfor} in a straightforward way, because a Delzant-type classification of symplectic toric Deligne--Mumford stacks is not known. In fact we can reconstruct a toric manifold from its moment polytope and the reconstruction is important to prove the above identity.

 However we have a combinatorial construction of toric Deligne--Mumford stacks: stacky fans \cite{borisov05:_chow_delig_mumfor} and stacky polytopes \cite{sakai13:_delig_mumfor}. While they do not give a classification for toric Deligne--Mumford stacks, they give a combinatorial construction of a wide class of Deligne--Mumford stacks. 

In this paper we study an extension of the theorem in the framework of stacky polytopes. The basic idea is to extend the theory of prequantisations to stacks. The stack associated to a stacky polytope has two special atlases. One gives a symplectic structure while another gives a holomorphic structure. The former is important to see the existence of prequantisations and the topology of the stack and the latter is needed to discuss holomorphic line bundles. Therefore we define a prequantisation without any atlases and discuss the relation between the two atlases.

We explain very briefly our main theorem (Theorem \ref{thm:main theorem}). A stacky polytope $(N,\Delta,\beta)$ consists of a $\Z$-module $N$, a polytope $\Delta \subset(N\otimes_\Z\R)^\vee$ and a homomorphism of $\Z$-modules $\beta:\Z^d \to N$ (Definition \ref{dfn:stacky polytope}). Here $d$ is the number of facets of $\Delta$. The $\Z$-module $N^\vee$ is the lattice in $(N\otimes_\Z\R)^\vee$. We can construct a symplectic Deligne--Mumford stack $\XD$ from the stacky polytope $(N,\Delta,\beta)$. We require the stacky polytope to satisfy two conditions (Assumption \ref{assumption: for stacky polytope}). The first condition is technical and the second condition guarantees the existence of a prequantisation (under the first condition). The main theorem states that 
 \[
 \dim H^0(\XD,\L) = 
 \begin{cases}
  \sharp(\Delta \cap N^{\!\vee}) & \text{if}\ c_\nu \in \Z\ (\nu = 1,\dots,d),
  \\
  0 & \text{else}.
 \end{cases}
 \]
Here $\L$ is the holomorphic line bundle associated to the prequantisation and $c_1,\dots,c_d$ are determined by the description (\ref{eqn:polytope-as-intersection-of-hyperplanes}) of the polytope $\Delta$.

In the case of toric manifolds, the condition ``$c_1,\dots,c_d \in \Z$'' is equivalent to the existence of prequantisation and the dimension of the space of holomorphic sections is always equal to the number of the lattice points of the polytope under the condition. However the condition is just a sufficient condition for the case of the stack $\XD$. 

This paper is organised as follows. In §\ref{sec:stacks-and-polytopes} we review briefly the theory of Deligne--Mumford (quotient) stacks and stacky polytopes. We introduce a prequantisation of a (DM quotient) stack in §\ref{sec:prequantisation}. We discuss existence and uniqueness for a prequantisation under several assumptions which holds for the stack associated to a stacky polytope. We also discuss an explicit condition for the existence of a prequantisation. After we give a holomorphic structure to the stack $\XD$ associated to the stacky polytope $(N,\Delta,\beta)$, we prove the main theorem in §\ref{sec:main section}. After we look at small examples, we discuss the relation between our main theorem and the Hirzebruch--Riemann--Roch theorem for stacks. 

\paragraph{Notation} 
Throughout this paper we use the following notation.

\begin{longtable}{ll}
 $\T$ & the 1-dimensional compact torus $\R/\Z$. We denote by $[\theta]$ the image of $\theta \in \R$ in $\T$. \\
 $\T_\C$ & the 1-dimensional complex torus $\C/\Z$ \\
 $\C^\times$ & the multiplicative group $\GL_1\C (= \C\setminus{0})$ \\
 $\vb e_1,\dots,\vb e_d$ & the standard basis of $\R^d$ \\ 
 $\vb e^1,\dots,\vb e^d$ & the dual basis of $\vb e_1,\dots,\vb e_d$ \\ 
 $\i$ & $\sqrt{-1}$ \\
 $z_1,\dots,z_d$ & the standard complex coordinate functions of $\C^d$ \\
 $X_1 \rightrightarrows X_0$ & the groupoid with $\Ob{X_1 \rightrightarrows X_0}=X_0$ and $\Ar{X_1 \rightrightarrows X_0}=X_1$ \\
 $BG$ & the classifying space of a Lie group $G$ \\
 $EG$ & a contractible space equipped with a free action of a Lie group $G$ \\
 $\g$ & the Lie algebra of the Lie group $G$ \\
 $e$ & the base of the natural logarithm (i.e. not the exponential map of a Lie group) \\
 $\pair{\ ,\ }$ & the natural pairing (i.e. $V^\vee \times V \to \R$ for a real vector space $V$.) \\
 $[\eta]$ & the de Rham cohomology class of a differential form $\eta$ \\
 $[\eta]_G$ & the cohomology class of $\eta$ in the Cartan model \\
 $\Z_G$ & the integral lattice: the kernel of $\exp:\g \to G$. ($G$ is a compact torus)\\
 $\Z_G^\vee$ & the weight lattice $\{w\in\g^\vee\mid\pair{w,\Z_G}\subset\Z\}$ ($G$ is a compact torus.)\\
\end{longtable}

 The notation ``${}^\vee$'' always means ``dual'' in the usual sense: $A^\vee := \Hom_{\Z}(A,\Z)$ for a $\Z$-module $A$ and $V^\vee := \Hom_{\R}(V,\R)$ for an $\R$-vector space $V$.


\section{Quotient stacks and stacky polytopes}
\label{sec:stacks-and-polytopes}

In this section we review briefly facts about quotient stacks and stacky polytopes \cite{sakai13:_delig_mumfor}.

\subsection{A differentiable stack}

Because we deal mainly with quotient stacks, we omit details of differentiable stacks. Details can be found in Behrend--Xu \cite{behrend11:_differ} or Metzler \cite{metzler03:_topol_smoot_stack} for example. For Lie groupoids, see Moerdijk \cite{moerdijk02:_orbif_group}, for example.

We denote by $\Diff$ the category of smooth manifolds and smooth maps. A stack (over $\Diff$) is a category fibred in groupoids over $\Diff$ which satisfies the descends condition. A stack consists of a category $\X$ and a functor $F_\X : \X \to \Diff$, however we will write it $\X$ for simplicity.

By definition, for each $U \in \Diff$ the fibre $\X_U$ over $U$ forms a groupoid. Here the classes of objects and arrows are defined by
\[
 \Ob{\X_U} = \{ x \in \X \mid F_\X(x)=U \}
 \quad\text{and}\quad
 \Ar{\X_U} = \{ a \in \Ar\X \mid F_\X(a)=\id_U \}
\]
respectively.

Stacks form a $(2,1)$-category $\St(\Diff)$. Namely for any stacks $\X$ and $\Y$ the class $\Hom(\X,\Y)$ of morphisms of stacks is a groupoid. Two stacks $\X$ and $\Y$ are said to be equivalent if there are two morphisms of stacks $f:\X \to \Y$ and $g:\Y \to \X$ such that $g \circ f \Leftrightarrow \id_\X$ in $\Hom(\X,\X)$ and $f \circ g  \Leftrightarrow \id_\Y$ in $\Hom(\Y,\Y)$.

By the 2-Yoneda lemma, the category $\Diff$ can be embedded into $\St(\Diff)$ as a full subcategory. Here a manifold $M$ is identified with the slice category $(\Diff \downarrow M)$. Thus $M_U$ for $U \in \Diff$ means the discrete category $\Cinf(U,M)$. A stack $\X$ is said to be representable if there is a manifold which is equivalent to $\X$.

A representable morphism $\phi$ from a manifold to a stack $\X$ is called an atlas of $\X$ if $\phi$ is a surjective submersion. Let $\phi:X_0 \to \X$ be an atlas and $X_1$ a manifold equivalent to the pullback $X_0 \times_\X X_0$. The Lie groupoid $X_1 \rightrightarrows X_0$ is said to be associated to the atlas $\phi:X_0 \to \X$.
\[
 \xymatrix@C=40pt{
 X_1 \ar[d]_\src \ar[r]^\tgt  & X_0 \ar[d]^\phi \\
 X_0 \ar[r]_\phi & \X 
 }
\]
The stack $\X$ is said to be differentiable if the Lie groupoid is proper: the map $(\src,\tgt):X_1 \to X_0 \times X_0$ is proper. This definition is independent of the choice of an atlas. A stack is said to be Deligne--Mumford (DM) if it admits an atlas such that the associated Lie groupoid is a proper étale groupoid. 

In this paper, we study the geometry and topology of a differentiable stack $\X$ through the Lie groupoid $X_1 \rightrightarrows X_0$ associated to an atlas. In particular, the homotopy type of $\X$ means that the homotopy type of the classifying space of the Lie groupoid. 

\subsection{A quotient stack}
\label{subsec:quatient stack}

Let $G$ be a Lie group. We mainly consider right $G$-manifolds, while we deal with a $G$-representation as a left $G$-module.



We define the quotient stack $[M/G]$ for a $G$-manifold $M$ as follows. The class of objects of the fibre $[M/G]_U$ over $U \in \Diff$ consists of pairs $(\pi^P,\epsilon^P)$ of a principal $G$-bundle $\pi^P:P \to U$ and a $G$-equivariant map $\epsilon^P$ from the total space $P$ to $M$:
\[
 \xymatrix{ U & P \ar[l]_(.4){\pi^P} \ar[r]^{\epsilon^P} & M}.
\]
An arrow from $(\pi^Q,\epsilon^Q)$ to $(\pi^P,\epsilon^P)$ is a pair $(f,\tilde f)$ of a map $f:V \to U$ and a $G$-equivariant map $\tilde f:Q \to P$ which make the diagram 
\[
 \xymatrix{
 U & P \ar[l]_{\pi^P} \ar[r]^{\epsilon^P} & M  \\
 V \ar[u]_f & Q \ar[l]^{\pi^Q}\ar[u]_{\tilde f}\ar[r]_{\epsilon^Q} & M \ar[u]_{\id_M}.
 }
\]
commutative.

The quotient stack $[M/G]$ has a natural atlas $\pi^M:M\to [M/G]$. The morphism $\pi^M$ assigns to each map $m:U \to M$ the diagram
\[
 \xymatrix@C=60pt{
 U & U \times G \ar[l]_{\pr_1} \ar[r]^{(u,g) \mapsto m(u) \cdot g} & M
 }.
\]
The Lie groupoid associated to the atlas is the action groupoid $M \times G \rightrightarrows M$. Therefore if the $G$-action on $M$ is proper, then the quotient stack $[M/G]$ is differentiable. Moreover $[M/G]$ is DM if the $G$-action is additionally locally free.

The stack $[\pt/G]$ is called the classifying stack of $G$ and denoted by $\BG$. There is a unique $G$-equivariant map $M \to \pt$ and this induces a morphism of stacks $F_M: [M/G] \to \BG$. The morphism $F_M$ is defined by 
\[
 \bigl(\xymatrix@C=20pt{U & P \ar[l]_(.4){\pi}\ar[r]^(.4){\epsilon} & M}\bigr)
 \mapsto
 \bigl(\xymatrix@C=20pt{U & P \ar[l]_(.4){\pi}\ar[r] & \pt}\bigr)
\]
for objects. 

Let $H$ be a closed subgroup of $G$. Every $G$-manifold $M$ is naturally an $H$-manifold and we have a morphism of stacks $\phi_M: [M/H] \to [M/G]$ defined by the formula 
\[
 \bigl(\xymatrix@C=20pt{U & Q \ar[l]_(.4){\pi^Q} \ar[r]^{\epsilon^Q} & M}\bigr)
 \mapsto
 \bigl(\xymatrix@C=20pt{U & Q \times_H G \ar[l]_(.55){\pi'} \ar[r]^(.6){\epsilon'} & M}\bigr)
\]
for objects. $\pi'$ and $\epsilon'$ are defined by $\pi'([q,g])=\pi^Q(q)$ and $\epsilon'([q,g]) = \epsilon^Q(q)g$ ($q \in Q$ and $g \in G$) respectively.

\begin{lem}
 \label{lem:basechange-diagram}
 The following diagram is 2-cartesian:
 \[
 \xymatrix{
 [M/H] \ar[r]^{\phi_M} \ar[d]_{F_M} & [M/G] \ar[d]^{F_M} \\
 \BH \ar[r]_{\phi_\pt} & \BG. \\
 }
 \]
\end{lem}
\begin{proof}
 The commutative diagram induces a morphism of stacks $[M/H] \to \BH \times_{\BG} [M/G]$. The inverse of the morphism is given by 
 \[
 \BH \times_{\BG} [M/G] \to [M/H] ;\
 \bigl(\pi^Q,(\pi^P,\epsilon^P),\alpha:Q \times_H G \to P\bigr)
 \mapsto
 \bigl(\pi^Q, q \mapsto (\epsilon^P\circ\alpha)([q,1])\bigr)
 \]
 for objects and an arrow $(a,b)$ in $\BH\times_{\BG} [M/G]$ is sent to the arrow $a$ in $[M/H]$.
\end{proof}


\subsection{Cohomology of a DM quotient stack}
\label{subsec:coh of a quot stack}

Since the classifying space of the action groupoid is (homotopically equivalent to) the Borel construction $EG \times_G M$ of the $G$-manifold $M$, the cohomology of the quotient stack $[M/G]$ is the equivariant cohomology of $M$. 

The complex of $G$-basic forms:
\[
 \Omegabas^*(M) := 
 \bigl\{\eta \in \Omega^*(M)^G \!\ \big|\!\
 \iota(\xi^M)\eta=0\ (\forall \xi \in \g) \bigr\}.
\]
is a subcomplex of $\Omega^*(M)$. Here $\xi^M$ is the infinitesimal action of $\xi$, i.e.
\[
 \xi^M(m) := 
 \left.\dfrac{d}{d\!\lambda}\right|_{\lambda=0} m \cdot \exp(\lambda\xi)
 \quad
 (m \in M).
\]
{\em If the $G$-action on $M$ is proper and locally free}, the cohomology ring of the complex of $G$-basic forms is isomorphic to the cohomology ring $H^*([M/G],\R)$. We call a $G$-basic form on $M$ a (differential) form on the quotient stack $[M/G]$.

{\em If $G$ is compact and connected}, then we can use the Cartan model for $H_G^*(M,\R)$ (i.e.\! $H^*([M/G],\R)$). We denote by $\Omega_G^k(M)$ the space of equivariant differential forms of degree $k$:
\[
 \Omega_G^k(M) := \bigoplus_{p+2q=k}\bigl(\Omega^p(M)\otimes S^q(\g^\vee)\bigr)^G.
\]
Here $S^q(\g^\vee)$ is the space of homogeneous polynomials on $\g$ of degree $q$ and inherits the $G$-module structure from the coadjoint representation $\g^\vee$. We regard an equivariant differential form $\eta$ as an equivariant polynomial function on $\g$ with values in differential forms. The graded algebra $\Omega_G^*(M)$ is equipped with the differential defined by the formula
\[
 (d_G \eta)(\xi):= d\eta(\xi) - \iota(\xi^M)\eta(\xi).
\]
The cohomology ring of the complex $(\Omega_G^*(M),d_G)$ is isomorphic to $H_G^*(M,\R)$. It is also known that the inclusion $\Omegabas^*(M) \to \Omega_G^*(M)$ is a quasi-isomorphism if the $G$-action is locally free.

The details of this discussion can be found in Behrend \cite{behrend04:_cohom} and Guillemin--Sternberg \cite{guillemin99:_super_rham}.

\subsection{A symplectic quotient stack}
\label{subsec:sympl-quot-stack}

Consider the standard linear $\T^d$-action on $\C^d$:
\[
 \C^d \acted \T^d:\
 (z_1,\dots,z_d) \cdot [\theta_1,\dots,\theta_d] :=
 (e^{2\pi\i\theta_1}z_1,\dots,e^{2\pi\i\theta_d}z_d).
\]
The function 
\[
 \mu_0: \C^d \to (\R^d)^\vee;\quad
 z \mapsto \pi \sum_{\nu=1}^d |z_\nu|^2 \vb e^\nu, 
\]
is a moment map of the $\T^d$-manifold $\C^d$. Namely the map $\mu_0$ is $\T^d$-equivariant and satisfies the identity $d\pair{\mu_0,\xi} = - \iota(\xi^M)\omegastd$ for all $\xi \in \R^d$. Here $\omegastd$ is the standard symplectic structure on $\C^d (= \R^{2d})$.

\begin{prop}
 [{\cite[Theorem 14]{sakai13:_delig_mumfor}}]
 \label{prop: a triple giving a symplectic DM stack}
 Let $(G,\rho,\tau)$ consist of 
 \begin{itemize}
  \item \squashup 
	a compact Lie group $G$ whose adjoint representation is trivial,
  \item \squashup 
	a Lie group homomorphism $\rho:G \to \T^d$, and
  \item \squashup 
	a regular value $\tau \in \g^\vee$ of the map $\mu := \rho^\vee \circ \mu_0$.
 \end{itemize}
 Here $\rho^\vee:(\R^d)^\vee \to \g^\vee$ is the dual of the induced Lie algebra homomorphism $\rho:\g \to \R^d$. If the level set $\mu^{-1}(\tau)$ is nonempty and the map $\mu:\C^d \to \g^\vee$ is proper, then the quotient stack $[\lm/G]$ is a compact symplectic DM stack. The symplectic form $\omega_\tau \in \Omegabas^2(\lm)$ is given by the restriction of $\omegastd$ to $\lm$.
\end{prop}

We denote by $\Xt$ the quotient stack and call $\Xt$ the stack associated to the triple $(G,\rho,\tau)$.

\subsection{A stacky polytope and the associated stack}
\label{subsec:stacky polytope}

We recall briefly the theory of stacky polytopes \cite[§3]{sakai13:_delig_mumfor}.

\begin{dfn}
 \label{dfn:stacky polytope}
 A triple $(N,\Delta,\beta)$ of 
\begin{itemize}
 \item \squashup 
       a finitely generated $\Z$-module or rank $r$, 
 \item \squashup
       a simple polytope $\Delta$ with $d$-facets $F_1,\dots,F_d$ in $(N \otimes_{\Z}\R)^\vee$ and 
 \item \squashup
       a homomorphism of $\Z$-modules $\beta:\Z^d \to N$ with finite cokernel
\end{itemize}
 is called a stacky polytope if the vectors $\beta(\vb e_1) \otimes 1,\dots,\beta(\vb e_d) \otimes 1$ in $N \otimes_\Z \R$ are perpendicular to the facets $F_1,\dots,F_d$ in inward-pointing way, respectively. 
\end{dfn}

We can assign to each stacky polytope $(N,\Delta,\beta)$ a triple $(G,\rho,\tau)$ satisfying the assumptions in Proposition \ref{prop: a triple giving a symplectic DM stack} as follows.

Take projective resolutions $E$ and $F$ of the $\Z$-modules $\Z^d$ and $N$ respectively. The homomorphism of $\Z$-modules $\beta:\Z^d \to N$ induces a map of chain complexes $\beta:E \to F$ and gives rise to a short exact sequence of chain complexes $0 \to F \to \mathrm{Cone}(\beta) \to E[1] \to 0$. Here $\mathrm{Cone}(\beta)$ is the mapping cone. The dual sequence $0 \to E[1]^\vee \to \mathrm{Cone}(\beta)^\vee \to F^\vee \to 0$ is a short exact sequence of cochain complexes and the long exact sequence of cohomology groups contains the exact sequence 
\begin{equation}
 \label{eqn:exact sequence for DG(beta)}
 \xymatrix@C=30pt{
  0 \ar[r] &
  N^\vee \ar[r]^(.4){\beta^\vee} &
  (\Z^d)^\vee \ar[r]^{\beta^\DG} &
  \DG(\beta) \ar[r] &
  \mathrm{Ext}_{\Z}^1(N,\Z) \ar[r] &
  0.
  }
\end{equation}
Here $\DG(\beta) := H^1\bigl(\mathrm{Cone}(\beta)^\vee\bigr)$. Applying the contravariant functor $\Hom_\Z(-,\T)$ to $\beta^\DG$, we obtain 
\[
 \rho: G \to \T^d,
\]
where $G:=\Hom_\Z(\DG(\beta),\T)$. Note that we can naturally identify $\Hom_\Z((\Z^d)^\vee,\T)$ with $\T^d$. Since $\DG(\beta)$ is a finitely generated $\Z$-module, $G$ is a compact abelian Lie group. Set $w^\nu := \rho^\vee(\vb e^\nu) \in \g^\vee$ ($\nu = 1,\dots,d$).

The second condition in Definition \ref{dfn:stacky polytope} implies that the polytope $\Delta$ is rational and can be described as
\begin{equation}
 \label{eqn:polytope-as-intersection-of-hyperplanes}
  \Delta = 
  \bigl\{ \eta \in (N \otimes_\Z \R)^\dual
  \big|
  \pair{\eta,\beta(\vb e_\alpha)\otimes 1} \geq -c_\alpha \bigr\}.
\end{equation}
for some $c_\alpha \in \R\ (\alpha = 1,\dots,d)$.

\begin{prop}%
 [{\cite[Theorem 18]{sakai13:_delig_mumfor}}]
 \label{prop:stacky polytope defines a symp DM stack}
 Let $(N,\Delta,\beta)$ be a stacky polytope. Define $(G,\rho,\tau)$ by the triple of 
 \begin{itemize}
  \item \squashup
	the Lie group $G = \Hom_\Z(\DG(\beta),\T)$,
  \item \squashup
	the homomorphism $\rho:G \to \T^d$ induced by $\beta^\DG:(\Z^d)^\vee \to \DG(\beta)$, and 
  \item \squashup
	the covector $\tau:= \sum_{\nu=1}^d c_\nu w^\nu$.
 \end{itemize}
 Here $c_1,\dots,c_d$ are the constants appearing in (\ref{eqn:polytope-as-intersection-of-hyperplanes}). Then the triple $(G,\rho,\tau)$ satisfies the assumptions of Proposition \ref{prop: a triple giving a symplectic DM stack}. 
\end{prop}

For a stacky polytope $(N,\Delta,\beta)$ we denote by $\XD$ the DM stack associated to the triple $(G,\rho,\tau)$ given by the above Proposition. We call $\XD$ the stack associated to the stacky polytope.

\section{Prequantisations}
\label{sec:prequantisation}

We introduce a prequantisation of a quotient stack. Its notion is mostly a straightforward extension of the notion of a prequantisation of a manifold. 

\subsection{Definitions of $\T$-bundles and connections}

Let $G$ be a Lie group and $M$ a $G$-manifold. We assume that the $G$-action on $M$ is proper and locally free so that the quotient stack $\X := [M/G]$ is DM. 

\begin{dfn}
 [{\cite[§4]{behrend11:_differ}}]
  \label{dfn:principal T-bundle over a stack}
 A principal $\T$-bundle over $\X$ is a representable morphism of stacks $\pi:\P \to \X$ together with a $2$-commutative diagram
 \begin{equation}
  \label{eqn:principal T-bundle over a stack}
   \vcenter{
   \xymatrix{
   \P \times \T \ar[d]_{\pr_1} \ar[r]^\Psi & \P \ar[d]^\pi \\
  \P \ar[r]_\pi & \X \\
  }
 }
 \end{equation} 
 such that for every $x: U \to \X$, the pullback $x^*\!\P \to U$ along $x$ is an ordinary principal $\T$-bundle over $U$.
\end{dfn}

By definition the pullback of $\P$ along the atlas $\pi^M:M \to \X$ is a principal $\T$-bundle $P \to M$. Namely we have a 2-cartesian diagram
\begin{equation}
 \label{pullback-of-principal-bundle-on-stack}
 \xymatrix{
 P \ar[d]_{} \ar[r]^{\pi^P} & \P \ar[d]^\pi \\
 M \ar[r]_{\pi^M}& \X.
 } 
\end{equation}
The manifold $P$ inherits a $G$-action from $M$ and we can see that $P \to M$ is a $G$-equivariant principal $\T$-bundle. 
 Conversely for a $G$-equivariant principal $\T$-bundle $P \to M$, the induced map $[P/G] \to \X$ gives us a principal $\T$-bundle on $\X$. The $\T$-action map $\psi:P \times \T \to P$ descends a $\T$-action map $\Psi:[P/G] \times \T \to [P/G]$ on $\P$. This correspondence gives rise to an equivalence between the category of principal $\T$-bundles on $\X$ and the category of $G$-equivariant principal $\T$-bundles on $M$ \cite[Proposition 4.3]{behrend11:_differ}. 

In the diagram (\ref{pullback-of-principal-bundle-on-stack}), the morphism of stack $\pi^P:P \to \P$ is an atlas of the stack $\P$ and induces an equivalence of stacks $[P/G] \to \P$ \cite[\S 3]{sakai12:_rieman}. Since the projection map $P \to M$ is a proper $G$-equivariant map to the locally free $G$-space, the $G$-action on $P$ is also proper and locally free. Therefore $\P$ is a DM stack and a differential form (§\ref{subsec:quatient stack}) and a vector field on $\P$ is meaningful. 

\begin{dfn}
 A connection of a principal $\T$-bundle $\pi:\P \to \X$ is a 1-form $\Theta$ on $\P$ satisfying the following conditions.
 \begin{enumerate}
  \item For any $\zeta \in \R (=\Lie \T)$, the identity $\Theta(\zeta^\P) = \zeta$ holds. 
  \item $\Theta$ is $\T$-invariant, i.e. $\Psi^*\Theta = \pr_1^*\Theta$ holds on $\Omega^1(\P \times \T)$. 
 \end{enumerate}
 Here morphisms of stacks $\Psi$ and $\pr_1$ in the second condition come from the diagram (\ref{eqn:principal T-bundle over a stack}), and the infinitesimal action $X^\P$ is a vector field on $\P$ which is defined by using the morphism of stacks $\Psi:\P \times \T \to \P$ \cite[§3.6]{lerman12:_hamil_delig_mumfor}.
\end{dfn}

We can explicitly describe the above definition in terms of atlases as follows. As we see above, the stack $\P$ is (equivalent to) the quotient stack $[P/G]$, where $P$ is the $G$-equivariant principal $\T$-bundle over $M$. Since the source map and the target map of the groupoid $P \times G \rightrightarrows P$ are both $\T$-equivariant, (the equivalent class of) the pair $\bigl((\zeta^P,0),\zeta^P\bigr)$ of vector fields on $P \times G$ and $P$ gives a vector field on $\P$. The connection $\Theta$ is a $G$-basic 1-form on $P$ and the pairing $\Theta(\zeta^\P)$ is given by the usual pairing on $P$: $\Theta(\zeta^\P) = \Theta(\zeta^P)$ and $\T$-invariance of $\Theta$ agrees with the usual $\T$-invariance. Therefore a connection $\Theta$ of a principal $\T$-bundle $\P \to \X$ is nothing but a $G$-basic $\T$-connection of the $G$-equivariant principal $\T$-bundle $P\to M$. 

\begin{remark}
 Our definition of a connection of the principal $\T$-bundle $\P\to\X$ agrees with definition of a connection of the principal $\T$-bundle $P \to M$ over the Lie groupoid $G \times M \rightrightarrows M$ \cite{laurent-gengoux07:_chern_weil}. A principal $\T$-bundle over a Lie groupoid might not admit a connection, but our principal bundle does. (See the proof of Theorem \ref{thm:existence of a prequantisation}.)
\end{remark}

We can easily see through atlases that the pullback $\pi^*\!: \Omega(\X) \to \Omega(\P)$ induces an isomorphism from $\Omega(\X)$ to the complex of $\T$-basic forms on $\P$. If $\Theta$ is a connection of the principal $\T$-bundle $\P \to \X$, then $d\Theta$ is a $\T$-basic 2-form on $\P$. Therefore there uniquely exists a 2-form $F_\Theta$ on $\X$ satisfying $\pi^*\!F_\Theta = d\Theta$. We call the closed 2-form $F_\Theta$ the curvature of the connection $\Theta$. The de Rham cohomology class $[F_\Theta]$ is the first Chern class of $\P$ \cite{laurent-gengoux07:_chern_weil}.

\subsection{A prequantisation of a quotient stack}

Let $G$ be a Lie group and $M$ a manifold equipped with a proper and locally free $G$-action. We fix a closed $2$-form $\omega$ on $\X:=[M/G]$ i.e.\! a closed $G$-basic $2$-form on $M$. If $\ker \omega_m$ coincides with $\{\xi^M(m) | \xi \in \g \}$ for every $m \in M$, then $\omega$ is a symplectic form on $\X$ in the sense of Lerman--Malkin \cite{lerman12:_hamil_delig_mumfor}. However we assume here only the closedness of $\omega$.

\begin{dfn}
 A prequantisation of $\X$ is a pair $(\P,\Theta)$ of a principal $\T$-bundle $\P \to \X$ and a connection $\Theta$ whose curvature $F_\Theta$ agrees with $\omega$.
\end{dfn}

\begin{remark}
 A prequantisation $(\P,\Theta)$ of $[M/G]$ is an equivariant prequantisation $(P,\Theta)$ of $(M,\omega,0)$, where $0:M \to \g^\vee$ is the ($G$-equivariant) zero map. The discussion about equivariant prequantisation depends largely on the Cartan model \cite{guillemin02:_momen_hamil}. Since we deal with a non-compact or non-connected Lie group $G$, we can not use the theory of equivariant prequantisations directly.
\end{remark}

\begin{assumption}
 \label{assumption:simple-connectedness and classification of line bundles}
 For the quotient stack $\X := [M/G]$, we assume the following conditions.
 \begin{enumerate} 
  \item \squashup
	The $G$-manifold $M$ is simply-connected.
  \item \squashup
	Principal $\T$-bundles on $\X$ are classified by $H^2(\X,\Z)$ through the first Chern class.
 \end{enumerate}
\end{assumption}

If $G$ is compact, then by restating the classification of smooth $G$-equivariant line bundles on $M$ \cite{guillemin02:_momen_hamil,riera01:_lifts}, we can see that Assumption (ii) holds for the stack $\Xt$ associated to the triple $(G,\rho,\tau)$ satisfying the assumptions in Proposition \ref{prop: a triple giving a symplectic DM stack}. 

\begin{remark}
 We do not assume $G$ to be compact, since we deal with a non-compact Lie group.
\end{remark}

\begin{thm}
 \label{thm:uniqueness of prequantisation}
 If $(\P_1,\Theta_1)$ and $(\P_2,\Theta_2)$ are both prequantisations of $\X$, then there is a map $\gamma:\X \to \T$ such that $\Theta_2 = \Theta_1 + \gamma^*\!d\!t$. Here $d\!t$ is the Maurer--Cartan form on $\T$.
\end{thm}

\begin{proof}
Since $[F_{\Theta_1}] = [\omega] = [F_{\Theta_2}]$, we may identify $\P_1$ with $\P_2$ by Assumption \ref{assumption:simple-connectedness and classification of line bundles} (ii) and denote by $\P$ each of them. 

The simple-connectedness of $M$ and the homotopy exact sequence of the fibration $M \times_GEG \to BG$ imply that $H^1(\X,\R) = 0$.
Since $\Theta_1-\Theta_2$ is a closed $\T$-basic 1-form, there is a function $f:\X \to \R$ satisfying $\pi^*\!d\!f = \Theta_1-\Theta_2$. The composition $\gamma$ of $f:\X \to \R$ and the quotient map $\R \to \T$ is what we want.
\end{proof}

\begin{remark}
 The reason why we take a function $\gamma:\X \to \T$ instead of $f:\X \to \R$ is that $\gamma$ gives a $G$- and $\T$-equivariant diffeomorphism $P \to P$ over $M$:
 \[
 P \to P;\ p \mapsto p \gamma'(\pi(p)).
 \]
 Here $\gamma': M \to \T$ is the composition of $\gamma$ and the atlas $M \to \X$. Since above map gives an $\T$-equivalence of stacks $\P \to \P$ over $\X$, Theorem \ref{thm:uniqueness of prequantisation} could say that a prequantisation of $\Xt$ is unique up to automorphisms.
\end{remark}

\begin{thm}
 \label{thm:existence of a prequantisation}
 Assume that $G$ is compact. Then $\X$ admits a prequantisation if and only if the cohomology class $[\omega]$ is integral.
\end{thm}

\begin{proof}
 If $\X$ admits a prequantisation $(\P,\Theta)$, then $[\omega]= c_1(\P) \in H^2(\X,\Z)$. 

 Conversely we assume that $[\omega] \in H^2(\X,\Z)$. By Assumption \ref{assumption:simple-connectedness and classification of line bundles} (ii), we have a $\T$-bundle $\P \to \X$ such that $c_1(\P) = [\omega]$. Let $P \to M$ be the $G$-equivariant principal $\T$-bundle associated to $\P$.

 Choose a $G$-invariant $\T$-connection $\theta \in \Omega^1(P)^G$ on $P$. (Such a connection can be obtained by averaging a $\T$-connection over $G$). Since the $G$-action on $P$ is locally free, we may have a $G$-connection $A$ on $P$ i.e.\! a $\g$-valued $G$-invariant 1-form $A$ on $P$ satisfying $\iota(\xi^M)A = \xi$ for all $\xi \in \g$. Define a 1-form $\Theta$ on $P$ by the formula 
\[
 \Theta(v) := \theta \bigl(v - A(v)^P(p)\bigr)
 \quad
 (p \in P,\ v \in T_pP)
\]
Then $\Theta$ is a $G$-basic $\T$-connection on $P$. Since $[F_\Theta] = c_1(\P) = [\omega]$, there is a $G$-basic $1$-form $\beta$ on $M$ satisfying $F^\Theta +d\!\beta = \omega$. Since $\pi^*\!\beta$ is $G$-basic and $\T$-basic, $\Theta+\pi^*\beta$ is also a $G$-basic $\T$-connection on $P$. Therefore the pair $(\P,\Theta+\pi^*\!\beta)$ is a prequantisation of $\X$.
\end{proof}

\begin{remark}
 From the stacky point of view, the assumption of Theorem \ref{thm:existence of a prequantisation} should be ``$\X$ admits a representable morphism $\X \to \BG$ for some compact Lie group $G$'', because the compactness of $G$ is not a property of $\X$ even though $\X$ is defined as the quotient stack $[M/G]$.
\end{remark}

\subsection{The integral condition}
\label{subsec:the integral condition}

In this subsection we consider a prequantisation of the stack associated to a triple $(G,\rho,\tau)$ with the following assumptions.
\begin{assumption}
 \label{assumption:simple-connectedness}
 For the triple $(G,\rho,\tau)$ satisfying the conditions in Proposition \ref{prop: a triple giving a symplectic DM stack}, we additionally assume that 
 \begin{itemize}
  \item \squashup
	The Lie group $G$ is a compact torus, and
  \item \squashup
	The level manifold $\lm$ is 2-connected.
 \end{itemize}
\end{assumption}

The above assumptions imply that the quotient stack $\Xt$ satisfies Assumption \ref{assumption:simple-connectedness and classification of line bundles}. The first condition allows us to use the Cartan model for equivariant cohomology. The second condition seems to be rather strong, but it holds for the triple $(G,\rho,\tau)$ defined by a stacky polytope.

\begin{lem}
 \label{lem:module map is an isom to integral classes}
 The module map $H^2(BG,\Z) \to H^2_G(\lm,\Z)$ is an isomorphism.
\end{lem}
\begin{proof}
 The long exact sequence associated to the fibre bundle $\lm\times_GEG\to BG$ implies that $\pi_1(\lm\times_GEG)$ vanishes and the induced homomorphism $\pi_2(\lm\times_GEG) \to \pi_2(BG)$ is an isomorphism. According to the Hurewicz theorem, $H_1(\lm\times_GEG,\Z)=0$ and the induced homomorphism $H_2(\lm\times_GEG,\Z) \to H_2(BG,\Z)$ is an isomorphism of free $\Z$-modules. This isomorphism implies the isomorphism of cohomology groups.
\end{proof}


\begin{lem}
 \label{lem:integer lattice is the lattice of integral classes}
 We can identify $\Z_G^\vee$ with $H^2_G(\lm,\Z)$.
\end{lem}
\begin{proof}
 For $\tau \in \Z^\vee_G$ we define the $G$-action on $\T$ by 
 \begin{equation}
  \label{eqn:G-equiv T-bundle over pt}
   \T \acted G;\ \ [t] \cdot \exp(\xi) := [t+\pair{\tau,\xi}]
 \end{equation}
 Here $[\ ]: \R \to \T = \R/\Z$ is the natural quotient map. We denote by $\T_\tau$ the $G$-space and we regard it as a $G$-equivariant principal $\T$-bundle over a point. It is known that the map 
\[
 \Z_G^\vee \to H^2_G(\pt,\Z);\ \
 \tau \mapsto c_1^G(\T_\tau \to \pt)
\]
is an isomorphism of abelian groups. Moreover the Maurer--Cartan form $d\!t$ on $\T_\tau$ is a $G$-invariant $\T$-connection and its equivariant curvature form is $\tau$, $c_1^G(\T_\tau\to\pt) = \tau$. 
\end{proof}

\begin{lem}
 \label{lem:tau is equiv coh class of omega-tau}
 The identity $[\omega_\tau]_G=\tau$ holds in the Cartan model.
\end{lem}
\begin{proof}
 It is easy to see that the $1$-form 
 $\lambda:= (4\i)^{-1}\sum_{\nu=1}^d\bigl(\bar z_\nu dz_\nu - z_\nu d\bar z_\nu \bigr)$
on $\C^d$ is a $G$-equivariant form and satisfies $d_G\lambda = \omegastd-\mu$. Its restriction to $\lm$ is $\omega_\tau-\tau$.
\end{proof}

By the two above lemmas and Theorem \ref{thm:existence of a prequantisation}, we can conclude the following theorem.

\begin{thm}
 \label{thm:integral condition for connected G}
 Under Assumption \ref{assumption:simple-connectedness}, 
 \[
 \tau \in \Z_G^\vee \iff
 [\omega_\tau]_G\ \text{is integral}\ \iff
 \Xt\ \text{admits a prequantisation}.
 \]
\end{thm}

\begin{remark}
 \label{remark:explicit description of principal T-bundle}
 Given $\tau \in \Z_G^\vee$, we can explicitly construct a principal $\T$-bundle on $\Xt$ under Assumption \ref{assumption:simple-connectedness}: The pullback $\lm \times \T_\tau$ of $\T_\tau \to \pt$ is a $G$-equivariant principal $\T$-bundle whose the $G$-equivariant first Chern class of is $\tau$. Here $\T_\tau$ is the $G$-space defined by the formula (\ref{eqn:G-equiv T-bundle over pt}). 

In terms of stacks the principal $\T$-bundle $\P:=[\lm\times\T_\tau/G] \to \Xt$ is a pullback of $[\T_\tau/G] \to \BG$ along the morphism $F_M:\Xt \to \BG$ (§\ref{subsec:quatient stack}). Note by Assumption \ref{assumption:simple-connectedness and classification of line bundles} (ii) that we may assume that every principal $\T$-bundle $\P\to\Xt$ is given by $[(\lm\times\T_\tau)/G] \to [\lm/G]$.
\end{remark}

\section{The number of the lattice points of a stacky polytopes}
\label{sec:main section}

In this section we discuss the relation between the number of the lattice points $\Delta \cap N^\vee$ of a stacky polytope $(N,\Delta,\beta)$ and prequantisation of the associated stack $\XD$ (§\ref{subsec:stacky polytope}). First we discuss a holomorphic atlas of $\XD$ so that we can deal with a holomorphic line bundle of $\XD$. We assign to each prequantisation $(\P,\Theta)$ of $\XD$ a holomorphic line bundle $\L\to\XD$ and define $\Q(\XD)$ by the dimension of the space of the holomorphic sections of $\L$. After making sure the condition of existence of a prequantisation in terms of a stacky polytope, we see the relation between $\Q(\XD)$ and $\sharp(\Delta \cap N^\vee)$. Finally we make a remark on the relation between $\Q(\XD)$ and the Hirzebruch--Riemann--Roch theorem.

\subsection{The holomorphic atlas of the stack associated to a stacky polytope}

We can translate the stacky polytope $(N,\Delta,\beta)$ into a stacky fan, which Borisov Chen and Smith \cite{borisov05:_chow_delig_mumfor} define, and the stacky fan gives us a holomorphic atlas of $\XD$. In this subsection we review the construction of the holomorphic atlas in terms of a stacky polytope \cite[§4]{sakai13:_delig_mumfor}.

Let $\C[z_1,\dots,z_d]$ be the coordinate ring of $\C^d$. Define the ideal $J_\Delta$ of $\C[z_1,\dots,z_d]$ by 
\[
 J_\Delta
 = \Bigl\langle
 \prod_{\nu:F_\nu \not\supset F}z_\nu
 \ \Big|\
  F\ \text{is a face of}\ \Delta
 \Bigr\rangle.
\]
Here $F_1,\dots,F_d$ are the facets of $\Delta$ as in Definition \ref{dfn:stacky polytope}. Denote by $Z_\Delta$ the complement of the algebraic set defined by the ideal $J_\Delta$:
\[
 Z_\Delta := \C^d\setminus\mathbb V(J_\Delta).
\]

Applying the contravariant functor $\Hom_\Z(-,\T_\C)$ to $\beta^\DG$ in the exact sequence (\ref{eqn:exact sequence for DG(beta)}), we obtain a homomorphism of Lie groups $\rho_\C:G_\C \to \T_\C^d$, where $G_\C := \Hom_\Z(\DG(\beta),\T_\C)$. The Lie group $G_\C$ acts on $\C^d$ through $\rho_\C$. We can see that $Z_\Delta$ is $G_\C$-invariant and $[Z_\Delta/G_\C]$ is a DM stack. 

\begin{prop}
 \label{prop:facts on the two equivalent stacks}
 \hfill 

 \vspace{-\intextsep}
 \begin{enumerate}
  \item \label{enu:hartogos}
	The affine open subset $Z_\Delta$ is the complement of the union of coordinate subspaces of (complex) codimension at least $2$.
  \item \label{enu:two-connectedness}
	Both $\lm$ and $Z_\Delta$ are 2-connected.
  \item \label{enu:equivalence}
	Define The morphism of stacks $\Phi:[\lm/G] \to [Z_\Delta/G_\C]$ by 
	\[
	\bigl(\xymatrix@C=15pt{
	U & P \ar[l]_(.34){\pi} \ar[r]^(.34){\epsilon} & \lm
	}\bigr)
	\mapsto
	\bigl(\xymatrix@C=15pt{
	U & P \times_{\lm}Z_\Delta \ar[l]_(.64){\Phi_\pi}\ar[r]^(.64){\Phi_\epsilon} & Z_\Delta
	}\bigr)
	\]
	for objects, where $\Phi_\pi = \pi\circ\pr_1$ and $P\times_{\lm}Z_\Delta$ and $\Phi_\epsilon$ are defined by the cartesian square
	\[
	\xymatrix{
	P \times_{\lm}Z_\Delta \ar[d] \ar[r]^(.6){\Phi_\epsilon} &
	Z_\Delta \ar[d] \\
	P \ar[r]_\epsilon &
	\lm.
	}
	\]
	Here $Z_\Delta \to \lm$ is the natural projection $Z_\Delta \to Z_\Delta/\!\exp(\i\g) = \lm$. The morphism $\Phi$ for arrows are naturally defined. Then $\Phi$ is an equivalence of stacks.
  \item \label{enu:basechange for associated stack}
	The diagram 
	\[
	\xymatrix{
	[\lm/G] \ar[r]^{\Phi} \ar[d]_{F_{\lm}} & [Z_\Delta/G_\C] \ar[d]^{F_{Z_\Delta}}\\
	\BG \ar[r]_{\phi_\pt} & \B G_\C
	}
	\]
	is 2-commutative. 
 \end{enumerate}
\end{prop}

The statement \ref{enu:hartogos} is a direct conclusion of the definition of $Z_\Delta$ and the statement \ref{enu:two-connectedness} is well-known. 
 The proof of the statement \ref{enu:equivalence} can be found in the author's paper \cite[§4]{sakai13:_delig_mumfor}. The statement \ref{enu:basechange for associated stack} follows from the definition of $\Phi$.

Since the $G_\C$-action on $Z_\Delta$ is holomorphic, the Lie groupoid $Z_\Delta \times G_\C \rightrightarrows Z_\Delta$ is a complex Lie groupoid and therefore the atlas $Z_\Delta \to \XD$ gives a holomorphic structure on $\XD$. The symplectic form $\omega_\tau$ of $\Xt$ is given by a $G_\C$-basic closed real $(1,1)$-form whose kernel coincides with $\{\xi^{Z_\Delta}(z)| \xi \in \Lie (G_\C)\}$ at each point $z \in Z_\Delta$ by definition. We use the same letter $\omega_\tau$ for the $2$-form on $Z_\Delta$.

\subsection{The number of the lattice points of a stacky polytope}

Let $(\P,\Theta)$ be a prequantisation for $\XD$. The pullback of $\P$ via the atlas $Z_\Delta \to \XD$ is a $G_\C$-equivariant principal $\T$-bundle $P' \to Z_\Delta$ and the pullback $\theta$ of the connection $\Theta$ is a $G_\C$-basic $\T$-connection on $P'$ whose curvature $F_\theta$ agrees with $\omega_\tau$. Since the $(0,2)$-component of $F_\theta$ ($=\omega_\tau$) vanishes, $\theta$ gives rise to a $G_\C$-equivariant holomorphic structure of the associated line bundle $L:= P' \times_{\!\T} \C \to Z_\Delta$. We call the DM stack $\L = [(P' \times_\T \C)/G_\C]$ the holomorphic line bundle associated to the prequantisation $(\P,\Theta)$.

From now on, we assume the following assumptions unless otherwise stated. 

\begin{assumption}
 \label{assumption: for stacky polytope}
 For a stacky polytope $(N,\Delta,\beta)$ we assume the following conditions.
 \begin{enumerate}
  \item \squashup
	The homomorphism $\beta:\Z^d \to N$ is surjective.
  \item \squashup
	$\tau = \sum_{\nu=1}^d c_\nu w^\nu \in \DG(\beta)$.
 \end{enumerate}
 Here the constants $c_1,\dots,c_d \in \R$ defined by the description of the polytope (\ref{eqn:polytope-as-intersection-of-hyperplanes}).
\end{assumption}

The first condition implies that the short exact sequence 
\[
 \xymatrix@C=20pt{
 0 \ar[r] & \ker(\beta) \ar[r] & \Z^d \ar[r]^\beta & N \ar[r] & 0}
\]
gives a projective resolution of $N$. We can see by direct calculation that $\DG(\beta)$ is isomorphic to $(\ker\beta)^\vee$. 
Since $\DG(\beta)$ is a free $\Z$-module, $G_\C = \Hom_\Z(\DG(\beta),\T_\C)$ is a complex torus and $G=\Hom_\Z(\DG(\beta),\T)$ is a compact torus. Therefore
the triple $(G,\rho,\tau)$ defined by the stacky polytope satisfies Assumption \ref{assumption:simple-connectedness} (and therefore Assumption \ref{assumption:simple-connectedness and classification of line bundles}).

The Lie algebra of $G$ is given by $\g = \Hom_\Z(\DG(\beta),\R)$ and the natural map $\R \to \T$ induces the exponential map $\g \to G$. Therefore the integral lattice is given by $\Z_G=\Hom(\DG(\beta),\Z)$. Note that the free $\Z$-module $\DG(\beta)$ is naturally embedded into $\g^\vee=\Hom_\R\bigl(\Hom_\Z(\DG(\beta),\R),\R\bigr)$:
\[
 \DG(\beta) \to \Hom_\R\bigl(\Hom_\Z(\DG(\beta),\R),\R\bigr);
 \quad w \mapsto \bigl(\xi \mapsto \xi(w)\bigr).
\]
It is easy to see that the weight lattice $\Z_G^\vee$ agrees with $\DG(\beta)$. Therefore, by Theorem \ref{thm:integral condition for connected G}, the second condition in Assumption \ref{assumption: for stacky polytope} is equivalent to the existence of a prequantisation for $\XD$ (under the first condition).

\begin{dfn}
 \label{dfn:dimension of the prequantisation space}
 Given a prequantisation $(\P,\Theta)$ of $\XD$, we define 
 \[
  \Q(\XD) := \dim H^0(\XD,\L).
 \]
 Here $\L$ is the holomorphic line bundle associated to the prequantisation.
\end{dfn}

Our main theorem is following:
\begin{thm}
 \label{thm:main theorem}
 If a stacky polytope $(N,\Delta,\beta)$ satisfies Assumption \ref{assumption: for stacky polytope}, then 
 \[
 \Q(\XD) = 
 \begin{cases}
  \sharp(\Delta \cap N^{\!\vee}) & \text{if}\ c_\nu \in \Z\ (\forall \nu),
  \\
  0 & \text{else}.
 \end{cases}
 \]
\end{thm}
In the rest of this subsection we prove this theorem.

\begin{remark}
 It is known that the above theorem holds for (compact) symplectic toric manifold. The subtle point is that the condition ``$c_\nu \in \Z\ (\forall \nu)$'' is just a sufficient condition of the existence of a prequantisation for the stack $\XD$. 
\end{remark}

By Assumption \ref{assumption: for stacky polytope}, the associated stack $\XD$ admits a prequantisation $(\P,\Theta)$. We have a convenient description of the holomorphic line bundle associated to a prequantisation. Let $\C_\tau$ be the $G_\C$-manifold defined by 
\[
 \C \acted G_\C;\ \
 v \cdot \exp(\xi) := v e^{2\pi\i\pair{\tau,\xi}} \quad
 \bigl(\xi \in \Lie(G_\C)\bigr).
\]
The $G_\C$-manifold $\C_\tau$ is naturally the $G$-manifold and the complex line bundle $\L$ satisfying $c_1(\L)=[\omega_\tau]$ is given by the pullback of $F_{\C_\tau}:[\C_\tau/G] \to \BG$ along $F_{\lm}:\XD = [\lm/G] \to \BG$. Lemma \ref{lem:basechange-diagram} and Proposition \ref{prop:facts on the two equivalent stacks} \ref{enu:basechange for associated stack} imply that $\L$ is also the pullback of $F_{\C_\tau}:[\C_\tau/G_\C] \to \B G_\C$ along $F_{Z_\Delta}:\XD \to \B G_\C$. Therefore $\L$ is equivalent to $[(Z_\Delta \times \C_\tau)/G]$: 
\[
 \xymatrix{
 Z_\Delta \times \C_\tau \ar[d] \ar[r] & 
 \L \ar[d] \ar[r] & 
 [\C_\tau/G_\C] \ar[d]  \\
 Z_\Delta \ar[r] & 
 \XD \ar[r] & 
 \B G_\C.
 }
\]
Here $Z_\Delta \times \C_\tau$ is equipped with the diagonal $G_\C$-action. The principal $\T$-bundle $\P$ can be described as the associated principal $\T$-bundle of $\L$. The $\T$-connection $\Theta$ defines a $G_\C$-equivariant holomorphic structure on the bundle $Z_\Delta \times \C_\tau \to Z_\Delta$, while it has the canonical $G_\C$-equivariant holomorphic structure. According to Assumption \ref{assumption:simple-connectedness and classification of line bundles} (ii) there is a $G_\C$-invariant gauge transformation which sends our holomorphic structure to the canonical one. Therefore the space of holomorphic sections $H^0(\XD,\L)$ is isomorphic to the $G_\C$-equivariant holomorphic sections of the $Z_\Delta \times \C_\tau \to Z_\Delta$ and 
\[
 \Q(\XD) = \dim \bigl\{\!\ \text{holomorphic}\ f:Z_\Delta \to \C_\tau \mid f(z \cdot g) = f(z) \cdot g \bigr\}.
\]

The following lemma completes the proof of the theorem.
\begin{lem}
 Let $\sigma:\R^d \to \Hom_\Z(N^\vee,\R)$ be the linear map induced by $\beta^\vee:N^\vee \to (\Z^d)^\vee$. We naturally identify $\Hom_\Z(N^\vee,\R)$ with $N\otimes_\Z\R$ so that the domain of the dual map $\sigma^\vee$ is $(N\otimes_\Z\R)^\vee$. 
Set $\tau'= \sum_{\nu=1}^d c_\nu \vb e^\nu$ and denote by $\Delta'$ the polytope $\sigma^\vee(\Delta) + \tau'$ in $(\R^d)^\vee$. Then the identity 
 \[
 \sharp(\Delta' \cap (\Z^d)^\vee) = \dim \bigl\{\!\ \text{holomorphic}\ f:Z_\Delta \to \C_\tau \mid f(z \cdot g) = f(z) \cdot g \bigr\}
 \]
 holds. 
\end{lem}

Note that $\beta^\vee:N^\vee \to (\Z^d)^\vee$ and $\sigma^\vee:(N\otimes_\Z\R)^\vee \to (\R^d)^\vee$ are both injective and $\beta^\vee$ coincides with the lattice map induced by $\sigma^\vee$. Therefore $\sigma^\vee + \tau'$ sends $\Delta$ to $\Delta'$ bijectively. If $c_\nu \in \Z$ for all $\nu$, then $\Delta \cap N^\vee$ is bijectively sent to $\Delta' \cap (\Z^d)^\vee$. If $c_\nu \not\in\Z$ for some $\nu$, then $\Delta'\cap (\Z^d)^\vee$ is empty.

We prove the above lemma by modifying the proof for toric manifolds \cite[§3]{hamilton08}. 

\begin{proof}
The Hartogs' theorem and Proposition \ref{prop:facts on the two equivalent stacks} \ref{enu:hartogos} imply that a $G_\C$-equivariant holomorphic map from $Z_\Delta$ to $\C_\tau$ extends to a $G_\C$-equivariant holomorphic map $f$ from $\C^d \to \C_\tau$. Consider its Taylor expansion
\[
 f(z) = \sum_{\alpha} A_\alpha z^\alpha
 \quad
 (A_\alpha \in \C).
\]
Here $\alpha \in \Z_{\geq 0}^d$ is a multi-index. We regard a multi-index $\alpha$ as an element of $(\Z^d)^\vee$ through the formula $\alpha = (\alpha_1,\dots,\alpha_d) = \alpha_1 \vb e^1 +\dots+ \alpha_d \vb e^d$. Then for $\xi \in \g_\C$ we have 
\begin{align*}
 f(z \cdot \exp(\xi)) \cdot \exp(-\xi)
 &=
 \Bigl(\sum_\alpha A_\alpha \prod_{\nu=1}^d \bigl(z_\nu e^{2\pi\i\pair{\vb e^\nu,\dot\rho_\C(\xi)}}\bigr)^{\alpha_\nu} \Bigr)
 e^{-2\pi\i\pair{\tau,\xi}} \\
 &= 
 \sum_\alpha A_\alpha z^\alpha e^{2\pi\i\pair{\sum_{\nu=1}^d \alpha_\nu \vb e^\nu,\dot\rho_\C (\xi)}}
 e^{-2\pi\i\pair{\tau,\xi}} \\
 &= 
 \sum_\alpha A_\alpha z^\alpha e^{2\pi\i\pair{\rho^\vee(\alpha)-\tau,\xi}}.
\end{align*}
Therefore the set 
\[
 \bigl\{z^\alpha \!\ \big|\!\ \alpha \in (\Z^d)^\vee,\ \pair{\alpha,\vb e_\nu} \geq 0\ (\forall \nu),\ \rho(\alpha)=\tau \bigr\}
\]
gives a $\C$-basis of the space of $G_\C$-equivariant holomorphic maps $\C^d \to \C_\tau$.

On the other hand the polytope $\Delta'$ is explicitly given by $  \Delta' = \{ s \in (\R^d)^\vee \mid \pair{s,\vb e_\nu} \geq 0\ (\forall \nu),\ \rho^\vee(s) = \tau \}$ \cite[Lemma 17]{sakai13:_delig_mumfor}. Thus $\Delta' \cap (\Z^d)^\vee = \bigl\{ \alpha \in (\Z^d)^\vee \big| \pair{\alpha,\vb e_\nu} \geq 0\ (\forall \nu),\ \rho^\vee(\alpha)=\tau \bigr\}$.
\end{proof}

\subsection{Examples}
\label{subsec:examples}

In this subsection we consider a stacky polytope which defines a weighted projective stack $\CP(a,ab)$. 

Let $a$ and $b$ be positive integers. Set $N = \Z_a \oplus \Z$, where $\Z_a := \Z/a\Z$. We denote by $\bar n$ for the image of $n \in \Z$ via the natural projection $\Z \to \Z_a$. We have a natural isomorphism $(N \otimes_\Z \R)^\vee \to \R^\vee$ and we identify $\R^\vee$ with $\R$. Set $\Delta = [0,c]$ ($c>0$). Define a homomorphism of $\Z$-module $\beta:\Z^2 \to N$ by 
\[
 \beta(\vb e_1) = (\bar 1, -b) \quad\text{and}\quad \beta(\vb e_2) = (\bar 0, 1).
\]
The triple $(N,\Delta,\beta)$ gives a stacky polytope and the polytope $\Delta$ is described by the form 
\[
 \Delta = \{ \eta \in \R \mid
 \pair{\eta,\beta(\vb e_1) \otimes 1} \geq -c_1,\
 \pair{\eta,\beta(\vb e_2) \otimes 1} \geq -c_2
 \}
\]
Here $c_1 = bc$ and $c_2 = 0$.

We construct explicitly the stack $\XD$ associated to the stacky polytope $(N,\Delta,\beta)$. Defining a homomorphism of $\Z$-modules $\dot\rho:\Z \to \Z^2$ by $1 \mapsto (a,ab)$, we obtain a short exact sequence of $\Z$-modules 
\[
 \xymatrix{
 0 \ar[r] & \Z \ar[r]^{\dot\rho} & \Z^2\ar[r]^\beta & N \ar[r] & 0 \\
 }.
\]
Thus $\DG(\beta) = \Z^\vee$ and therefore $G = \Hom_\Z(\DG(\beta),\T) = \T$. Applying the contravariant functor $\Hom_\Z(-,\T)$, we have the short exact sequence of tori
\[
 \xymatrix{
 0 \ar[r] &
 G \ar[r]^{\rho} &
 \T^2 \ar[r]^(.3){\sigma} & 
 \Hom_\Z(N^\vee,\T) \ar[r] & 
 0.
}
\]
The weights of the $G$-action are $w^1=\rho^\vee(\vb e^1) =a$ and $w^2=\rho^\vee(\vb e^2) =ab$. The moment map of the $G$-action on $\C^2$ is given by 
\[
 \mu:\C^2 \to \R; \ (z_1,z_2) \mapsto \pi a|z_1|^2 + \pi ab|z_2|^2
\]
and $\tau = c_1w^1 + c_2w^2 = abc$ is a regular value of $\mu$. Therefore the stack $[\lm/G]$ is the weighted projective stack $\CP(a,ab)$.

Now Assumption \ref{assumption: for stacky polytope} (i) holds while (ii)---the existence of prequantisation---is equivalent to $abc \in \Z$. 

\begin{enumerate}
 \item $\CP(1,2)$ ($a=1$, $b=2$)

       The stack $\XD$ admits a prequantisation if and only if $2c \in \Z$.

       \begin{tabular}[t]{c|cccccccc}
	$c$       & 1/2 & 1 & 3/2 & 2 & 5/2 & 3 & 7/2 & 4 \\ \hline
	$\tau$    & 1   & 2 & 3   & 4 &   5 & 6 & 7   & 8 \\ \hline
	$\Q(\XD)$ & 1   & 2 & 2   & 3 & 3   & 4 & 4   & 5
       \end{tabular}
 \item $\CP(1,3)$ ($a=1$, $b=3$)

       The stack $\XD$ admits a prequantisation if and only if $3c \in \Z$.

       \begin{tabular}[t]{c|ccccccccc}
	$c$       & 1/3 & 2/3 & 1 & 4/3 & 5/3 & 2 & 7/3 & 8/3 \\ \hline
	$\tau$    & 1   & 2   & 3 & 4   & 5   & 6 & 7   & 8   \\ \hline
	$\Q(\XD)$ & 1   & 1   & 2 & 2   & 2   & 3 & 3   & 3
       \end{tabular}

 \item $\CP(1,b)$ ($a=1$)

       The stack $\XD$ admits a prequantisation if and only if $bc (=\tau) \in \Z$ and we have 
       \[
       \Q(\XD)
       = \lfloor c \rfloor + 1 
       = \Biggl\lfloor \frac{\tau}{b} \Biggr\rfloor +1.
       \]
       Here $\lfloor \ \rfloor$ is the floor function.

 \item $\CP(2,2)$ ($a=2$, $b=1$)

       The stack $\XD$ admits a prequantisation if and only if $2c (=\tau) \in \Z$. However $c_1 \in \Z$ is equivalent to $c \in \Z$ (i.e. $\tau \in 2\Z$). Therefore if $c\not\in\Z$, then $\Q(\XD)=0$ even though $\XD$ admits a prequantisation.

       \begin{tabular}[t]{c|ccccccccc}
	$c$       & 1/2 & 1 & 3/2 & 2 & 5/2 & 3 & 7/2 & 4 \\ \hline
	$\tau$    & 1   & 2 & 3   & 4 & 5   & 6 & 7   & 8 \\ \hline
	$\sharp(\Delta \cap N^\vee)$ & 1   & 2 & 2   & 3 & 3   & 4 & 4   & 5 \\ \hline
	$\Q(\XD)$ & 0   & 2 & 0   & 3 & 0   & 4 & 0   & 5
       \end{tabular}

 \item $\CP(3,3)$ ($a=3$, $b=1$)

       The stack $\XD$ admits a prequantisation if and only if $3c (=\tau) \in \Z$. However $c_1 \in \Z$ is equivalent to $c \in \Z$ (i.e. $\tau \in 3\Z$). Therefore if $c\not\in\Z$, then $\Q(\XD)=0$ even though $\XD$ admits a prequantisation.

       \begin{tabular}[t]{c|ccccccccc}
	$c$       & 1/3 & 2/3 & 1 & 4/3 & 5/3 & 2 & 7/3 & 8/3 \\ \hline
	$\tau$    & 1   & 2   & 3 & 4   & 5   & 6 & 7   & 8   \\ \hline
	$\sharp(\Delta \cap N^\vee)$ & 1 & 1 & 2 & 2 & 2 & 3 & 3 & 3 \\ \hline
	$\Q(\XD)$ & 0   & 0 & 2   & 0 & 0   & 3 & 0   & 0
       \end{tabular}

 \item $\CP(a,a)$ ($b=1$)

       The stack $\XD$ admits a prequantisation if and only if $ac (=\tau) \in \Z$. However $c_1 \in \Z$ is equivalent to $c \in \Z$. Therefore if $c \not\in\Z$, then $\Q(\XD)=0$ even though $\XD$ admits a prequantisation. If $c\in\Z$ (i.e. $\tau \in a\Z$), then 
       \[
       \Q(\XD) = c + 1 = \frac{\tau}{a} +1.
       \]
\end{enumerate}

\subsection{Relation between $\Q(\XD)$ and the Hirzebruch--Riemann--Roch theorem}

In the rest of the paper we stick with the complex analytic atlas $Z_\Delta \to \XD$ and a holomorphic line bundle $\L \to \XD$ so that we can regard $\XD$ as a Deligne--Mumford stack over the category of schemes over $\C$ and $\L$ as an algebraic line bundle. 

In the theory of (ordinary) quantisations, we define the quantisation space by the virtual vector space 
\[
 \sum_q (-1)^q H^q(\XD,\L).
\]
The Euler characteristic $\chi(\XD,\L)$ of $\L$ is the dimension of the virtual vector space by definition. 

If $\XD$ is representable (i.e. $\XD$ is a symplectic toric manifold), then we can see $H^q(\XD,\L)=0$ for $q>0$ by applying the Demazure vanishing theorem. Therefore $\Q(\XD)=\chi(\XD,\L)$ holds. Unfortunately the Demazure vanishing theorem for the toric stack $\XD$ has not been established yet.

Applying the Kodaira vanishing theorem \cite{matsuki05:_kawam_viehw_kodair}, we can however see the identity $\Q(\XD)=\chi(\XD,\L)$ for many cases. In particular if $\XD$ is Fano (e.g. a weighted projective stack), then the identity holds and we can apply the Hirzebruch--Riemann--Roch theorem for stacks \cite{edidin12:_rieman_roch_delig_mumfor} to compute $\Q(\XD)$.

We check directly for the weighted projective stacks $\CP(a,ab)$ ($a=1$ or $b=1$) that $\Q(\XD)$ agrees with the Euler characteristic of $\L$. See §\ref{subsec:examples} for the notation for the stack and see Edidin \cite{edidin12:_rieman_roch_delig_mumfor} for the notation relating to the $K$-theory and the Hirzebruch--Riemann--Roch theorem.

Denote by $\ell$ the $G_\C$-equivariant line bundle $\C_1 \to \pt$. Then the representation ring $R(\C^\times)$ ($=K_0(\C^\times\!,\pt)$) is isomorphic to $\Z[\ell^{\pm}]$ and the bundle $\C_\tau \to \pt$ is $\ell^\tau$ in $R(\C^\times)$ ($\tau \in \Z = \Z^\vee$). The terminal map $Z_\Delta \to \pt$ induces the module map $R(G_\C) \to K_0(\C^\times\!,\C^2\setminus\{0\}) = K_0(\XD)$ and $[\L] = \ell^\tau$ in $K_0(\XD)$.

Applying the localisation theorem \cite[Theorem 2.7]{thomason87:_algeb_k}, we can see that the module map is surjective and its kernel is generated by $(\ell^{a}-1)(\ell^{ab}-1)$. Thus $K_0(\XD) \cong \Z[\ell^\pm]/\pair{(\ell^{a}-1)(\ell^{ab}-1)}$. The algebra $K_0(\XD)\otimes\C$ is supported at the points $1,\zeta,\zeta^2,\dots,\zeta^{ab-1}$ of $\C^\times = \Spec(R(\C^\times)\otimes\C)$, where $\zeta$ is a primitive $ab$-th root of unity.


The case of $\CP(1,b)$ ($a=1$) we have 
\begin{align*}
 \chi(\XD,\ell^\tau)
 &= 
 \int_{\CP(1,b)} \ch (\ell^\tau) \Td(\CP(1,b))
 +
 \sum_{k=1}^{b-1} \int_{\CP(b)} 
 \frac{\ch(\zeta^{k\tau}\ell^\tau)}{\ch(1-(\zeta^k\ell)^{b-1})}
 \Td(\CP(b))
 \\
 &=
 \frac{2\tau + 1 + b}{2b}
 +
 \frac{1}{b}\sum_{k=1}^{b-1} \frac{\zeta^{k\tau}}{1-\zeta^{k(b-1)}}
 \\
 &=
 \Biggl\lfloor\frac{\tau}{b}\Biggr\rfloor + 1.
\end{align*}

The case of $\CP(a,a)$ ($b=1$), we have 
\begin{align*}
 \chi(\XD,\ell^\tau)
 &= \sum_{k=0}^{a-1}
 \int_{\CP(a,a)} \ch((\zeta^k\ell)^\tau)\Td(\CP(a,a))
 = \sum_{k=0}^{a-1}
 \zeta^{k\tau}  \frac{\tau+a}{a^2}
 =
 \begin{cases}
  \frac{\tau}{a}+1 & \tau \in a\Z, \\
  0 & \tau \not\in a\Z.
 \end{cases}
\end{align*}



\paragraph{Acknowledgement}
The author thanks Reyer Sjamaar for raising the problem about the Danilov's theorem for toric stacks several years ago. The author also thanks Benjamin Collas for many discussions about Deligne--Mumford stacks. This work is supported by Professor Michael Weiss' Humboldt Professorship.


\medskip
{\em
\noindent
Mathematisches Institut, WWU Münster \newline
Einsteinstrasse 62 \newline
48149 Münster \newline
GERMANY

\noindent
\tt E-mail: sakai@blueskyproject.net
}
\end{document}